\documentclass[12pt,twoside,reqno]{amsart}

\usepackage{version}         

 \usepackage{color}

\usepackage{multirow}
\usepackage[OT1]{fontenc}
\usepackage{type1cm}
\usepackage{amssymb}
\usepackage{mathrsfs}
\usepackage[english]{babel}
\usepackage[latin1]{inputenc}
\usepackage[T1]{fontenc}
\usepackage{palatino}
\usepackage{amsfonts}
\usepackage{amsmath}
\usepackage{amssymb}
\usepackage{amsthm}
\usepackage{bbm}
\usepackage{extarrows}
\usepackage{graphicx}
\usepackage[usenames,dvipsnames]{xcolor}
\usepackage{wrapfig}
\usepackage{type1cm}
\usepackage[bf]{caption}
\usepackage{esint}
\usepackage{version}
\usepackage[colorlinks = true, citecolor = black, linkcolor = black, urlcolor = black, pdfstartview=FitH]{hyperref}
\usepackage{caption}
\usepackage{tikz}
\usepackage{bbding, wasysym}
\usepackage{enumitem}

\usepackage[left=2.3cm,top=3cm,right=2.3cm]{geometry}
\numberwithin{equation}{section}

\theoremstyle{plain}
\newtheorem{theorem}{Theorem}[section]

\newtheorem{proposition}[theorem]{Proposition}
\newtheorem{lemma}[theorem]{Lemma}
\newtheorem{notation}[theorem]{Notation}

\theoremstyle{remark}
\newtheorem{remark}[theorem]{Remark}

\newtheorem*{ack}{Acknowledgement}

\theoremstyle{definition}
\newtheorem{definition}[theorem]{Definition}

\newcommand{\R}{\mathbb{R}}

\newcommand{\Q}{\mathbb{Q}}
\newcommand{\N}{\mathbb{N}}

\newcommand{\cB}{\mathcal{B}}

\newcommand{\cH}{\mathcal{H}}
\newcommand{\cN}{\mathcal{N}}

\newcommand{\cU}{\mathcal{U}}

\newcommand{\cF}{\mathcal{F}}

\newcommand{\cG}{\mathcal{G}}
\newcommand{\cC}{\mathcal{C}}

\newcommand{\cA}{\mathcal{A}}

\renewcommand{\emptyset}{\varnothing}
\renewcommand{\epsilon}{\varepsilon}
\renewcommand{\rho}{\varrho}
\renewcommand{\phi}{\varphi}

\begin{document}

\title[Hausdorff dimension of frequency sets of univoque sequences]{Hausdorff dimension of frequency sets of univoque sequences}

\author{Yao-Qiang Li}

\address{Institut de Math\'ematiques de Jussieu - Paris Rive Gauche \\
         Sorbonne Universit\'e - Campus Pierre et Marie Curie\\
         Paris, 75005 \\
         France}

\email{yaoqiang.li@imj-prg.fr\quad yaoqiang.li@etu.upmc.fr}

\address{School of Mathematics \\
         South China University of Technology \\
         Guangzhou, 510641 \\
         P.R. China}

\email{scutyaoqiangli@gmail.com\quad scutyaoqiangli@qq.com}

\date{\today}

\begin{abstract}
For integer $m\ge3$, we study the dynamical system $(\Lambda_m,\sigma_m)$ where $\Lambda_m$ is the set $\{w\in\{0,1\}^\N: w$ does not contain $0^m$ or $1^m\}$ and $\sigma_m$ is the shift map on $\{0,1\}^\N$ restricted to $\Lambda_m$, study the Bernoulli-type measures on $\Lambda_m$ and find out the unique equivalent $\sigma_m$-invariant ergodic probability measure. As an application, we obtain the Hausdorff dimension of the set of univoque sequences, the Hausdorff dimension of the set of sequences in which the lengths of consecutive $0$'s and consecutive $1$'s are bounded, and the Hausdorff dimension of their frequency subsets.
\end{abstract}

\maketitle

\section{Introduction}

Let $\N$ be the set of positive integers $\{1,2,3,\cdots\}$ and define
$$\Gamma:=\Big\{w\in\{0,1\}^\N: \overline{w}<\sigma^kw<w \text{ for all } k\ge1\Big\}$$
where $\sigma$ is the shift map on $\{0,1\}^\N$, $\overline{0}:=1$, $\overline{1}:=0$ and $\overline{w}:=\overline{w}_1\overline{w}_2\cdots$ for all $w=w_1w_2\cdots\in\{0,1\}^\N$.

The set $\Gamma$ is strongly related to two well known research topics, iterations of unimodal functions and unique expansions of $1$ (see \cite{AC01} for more details).

On the one hand, in 1985, Cosnard \cite{C85} proved that a binary sequence $\alpha=(\alpha_n)_{n\ge1}$ is the kneading sequence of $1$ for some unimodal function $f$ if and only if $\tau(\alpha)\in\Gamma'$, where $\tau:\{0,1\}^\N\to\{0,1\}^\N$ is a bijection defined by $\tau(w):=(\sum_{i=1}^nw_i$ $($mod $2))_{n\ge1}$ and
$$\Gamma':=\Big\{w\in\{0,1\}^\N: \overline{w}\le\sigma^kw\le w \text{ for all } k\ge0\Big\}$$
is similar to $\Gamma$ in the sense that $\Gamma'\setminus\{$periodic sequences$\}=\Gamma$. The structure of $\Gamma'\setminus\{(10)^\infty\}$ was studied in detail by Allouche \cite{A83} (see also \cite{AC83}). The generalizations of $\Gamma$ and $\Gamma'$ (to more than two digits) are studied in \cite{A83,AF09}.

On the other hand, expansions of real numbers in non-integer bases were introduced by R\'enyi \cite{R57} in 1957 and then widely studied until now (see for examples \cite{ACS09,B89,BW14,KL98,LW08,LL18,P60,S97}). In 1990, Erd\"os, Jo\'o and Komornik \cite{EJK90} proved that a sequence $\alpha=(\alpha_n)_{n\ge1}\in\{0,1\}^\N$ is the unique expansion of $1$ in some base $q\in(1,2)$ if and only if $\alpha\in\Gamma$. Thus we call $\Gamma$ the set of \textit{univoque sequences} in this paper.

For any $p\in[0,1]$, the \textit{frequency subsets} of $\Gamma$ are defined by
$$\Gamma_p:=\Big\{w\in\Gamma:\lim_{n\to\infty}\frac{\sharp\{k:1\le k\le n,w_k=0\}}{n}=p\Big\},$$
$$\underline{\Gamma}_p:=\Big\{w\in\Gamma:\varliminf_{n\to\infty}\frac{\sharp\{k:1\le k\le n,w_k=0\}}{n}=p\Big\},$$
$$\overline{\Gamma}_p:=\Big\{w\in\Gamma:\varlimsup_{n\to\infty}\frac{\sharp\{k:1\le k\le n,w_k=0\}}{n}=p\Big\},$$
and the \textit{frequency subsets} of
$$\Lambda:=\Big\{w\in\{0,1\}^\N:\text{the lengths of consecutive } 0\text{'s and consecutive }1\text{'s in }w\text{ are bounded}\Big\}$$
are defined by
$$\Lambda_p:=\Big\{w\in\Lambda:\lim_{n\to\infty}\frac{\sharp\{k:1\le k\le n,w_k=0\}}{n}=p\Big\},$$
$$\underline{\Lambda}_p:=\Big\{w\in\Lambda:\varliminf_{n\to\infty}\frac{\sharp\{k:1\le k\le n,w_k=0\}}{n}=p\Big\},$$
$$\overline{\Lambda}_p:=\Big\{w\in\Lambda:\varlimsup_{n\to\infty}\frac{\sharp\{k:1\le k\le n,w_k=0\}}{n}=p\Big\},$$
where $\sharp$ denotes the cardinality. Let
$$\cU:=\Big\{q\in(1,2): \text{ the }q\text{-expansion of }1\text{ is unique}\Big\}$$
be the set of \textit{univoque bases}. It is proved in \cite{DK95,KKL17} that $\cU$ is of full Hausdorff dimension (see \cite{F90} for definition). That is,
$$\dim_H\cU=1.$$

On frequency sets, there is a well known classical result given by Eggleston \cite{E49} saying that for any $p\in[0,1]$,
$$\dim_H\Big\{x\in[0,1):\lim_{n\to\infty}\frac{\sharp\{k:1\le k\le n,\epsilon_k(x)=0\}}{n}=p\Big\}=\frac{-p\log p-(1-p)\log(1-p)}{\log2}$$
where $\epsilon_1(x)\epsilon_2(x)\cdots\epsilon_n(x)\cdots$ is the \textit{greedy binary expansion} of $x$, and $0\log0:=0$.

Correspondingly, we give the following as the main result in this paper.

\begin{theorem}\label{main}
Let $\pi_2$ be the natural projection from $\{0,1\}^\N$ to $[0,1]$, $\dim_H$ denote the Hausdorff dimension in $[0,1]$ and $\dim_H(\cdot,d_2)$ denote the Hausdorff dimension in $\{0,1\}^\N$ equipped with the usual metric $d_2$. Then
\begin{itemize}
\item[\emph{(1)}] $\dim_H\pi_2(\Gamma)=\dim_H(\Gamma,d_2)=\dim_H\pi_2(\Lambda)=\dim_H(\Lambda,d_2)=1$;
\item[\emph{(2)}] for any $p\in[0,1]$, we have
$$\begin{aligned}
&\dim_H\pi_2(\Gamma_p)=\dim_H\pi_2(\underline{\Gamma}_p)=\dim_H\pi_2(\overline{\Gamma}_p)\\
=&\dim_H(\Gamma_p,d_2)=\dim_H(\underline{\Gamma}_p,d_2)=\dim_H(\overline{\Gamma}_p,d_2)\\
=&\dim_H\pi_2(\Lambda_p)=\dim_H\pi_2(\underline{\Lambda}_p)=\dim_H\pi_2(\overline{\Lambda}_p)\\
=&\dim_H(\Lambda_p,d_2)=\dim_H(\underline{\Lambda}_p,d_2)=\dim_H(\overline{\Lambda}_p,d_2)=\frac{-p\log p-(1-p)\log(1-p)}{\log 2}.
\end{aligned}$$
\end{itemize}
\end{theorem}

This paper follows a similar framework and method as \cite{LLS19}, but most of the details are different. In Section 2, we define the symbolic dynamical system $(\Lambda_m,\sigma_m)$, and give some basic notations and preliminaries on dynamical systems and measure theory. In Section 3, we study some related digit occurrence parameters and their properties which will be used later. In Section 4, we define the Bernoulli-type measure $\mu_p$ on $(\Lambda_m,\cB(\Lambda_m))$, where $\cB(\Lambda_m)$ is the Borel sigma-algebra on $\Lambda_m$ (equipped with the usual metric). The measure $\mu_p$ turns out to be not $\sigma_m$-invariant, but we find out that there exists a unique $\sigma_m$-invariant ergodic probability measure on $(\Lambda_m,\cB(\Lambda_m))$ equivalent to $\mu_p$. In Section 5, we study some basic properties of Hausdorff dimension in general metric spaces. Finally in Section 6, we apply the Bernoulli-type measures in Section 4 and the properties of Hausdorff dimension in Section 5 to obtain our main result.

Throughout this paper, we use $\N$, $\Q$ and $\R$ to denote the sets of positive integers, rational numbers and real numbers respectively.

\section{Notations and preliminaries}

Let $\{0,1\}^*:=\bigcup_{n=1}^\infty\{0,1\}^n$ and $\{0,1\}^\N$ be the sets of finite words and infinite sequences respectively on two digits $\{0,1\}$. For any integer $m\ge3$, define
$$\Lambda_m:=\Big\{w\in\{0,1\}^\N: w \text{ does not contain } 0^m \text{ or } 1^m\Big\},$$
$$\Lambda_m^*:=\Big\{w\in\{0,1\}^*: w \text{ does not contain } 0^m \text{ or } 1^m\Big\}$$
and
$$\Lambda_m^n:=\Big\{w\in\{0,1\}^n: w \text{ does not contain } 0^m \text{ or } 1^m\Big\}$$
where $n\in\N$. For a finite word $w\in\{0,1\}^*$, we use $|w|$, $|w|_0$ and $|w|_1$ to denote its length, the number of $0$'s in $w$ and the number of $1$'s in $w$ respectively. Besides, $w|_k:=w_1w_2\cdots w_k$ denotes the prefix of $w$ with length $k$ for $w\in\{0,1\}^\N$ or $w\in\{0,1\}^n$ where $n\ge k$.

Let $\sigma:\{0,1\}^\N\to\{0,1\}^\N$ be the \textit{shift map} defined by
$$\sigma(w_1w_2\cdots)=w_2w_3\cdots\quad\text{for }w\in\{0,1\}^\N$$
and $d_2$ be the \textit{usual metric} on $\{0,1\}^\N$ defined by
$$d_2(w,v):=2^{-\inf\{k\ge0: w_{k+1}\neq v_{k+1}\}}\quad\text{for }w,v\in\{0,1\}^\N,$$
where $2^{-\infty}=0$. Then $\sigma$ is continuous on $(\{0,1\}^\N,d_2)$. Noting that $\sigma(\Lambda_m)=\Lambda_m$, we use $\sigma_m:\Lambda_m\to\Lambda_m$ to denote the restriction of $\sigma$ on $\Lambda_m$. Then $(\Lambda_m,\sigma_m)$ is a dynamical system. It is straightforward to check that the \textit{natural projection map} $\pi_2:\{0,1\}^\N\rightarrow[0,1]$, defined by
$$\pi_2(w):=\sum_{n=1}^\infty\frac{w_n}{2^n} \quad \text{for } w \in \{0,1\}^\N,$$
is surjective and continuous. Besides, we need the following concepts.

\begin{definition}[Cylinder]\label{Cylinder}
Let $m\ge3$ be an integer and $w\in\Lambda_m^*$. We call
$$[w]:=\Big\{v\in\Lambda_m:v\text{ begins with }w\Big\}$$
the \textit{cylinder} in $\Lambda_m$ generated by $w$.
\end{definition}

\begin{definition}[Absolute continuity and equivalence]
Let $\mu$ and $\nu$ be measures on a measurable space $(X,\cF)$. We say that $\mu$ is \textit{absolutely continuous} with respect to $\nu$ and denote it by $\mu \ll \nu$ if, for any $A \in\cF$, $\nu(A) = 0$ implies $\mu(A) = 0$. Moreover, if $\mu\ll\nu$ and $\nu\ll\mu$ we say that $\mu$ and $\nu$ are \textit{equivalent} and denote this property by $\mu\sim\nu$.
\end{definition}

\begin{definition}[Invariance and ergodicity]
Let $(X, \mathcal{F}, \mu, T)$ be a measure-preserving dynamical system, that is, $(X,\cF,\mu)$ is a probability space and $\mu$ is $T$-\textit{invariant}, i.e., $T\mu = \mu$ (i.e., $\mu\circ T^{-1}=\mu$). We say that the probability measure $\mu$ is \textit{ergodic} with respect to $T$ if for every $A\in\cF$ satisfying $T^{-1}A=A$ (such a set is called $T$-\textit{invariant}), we have $\mu(A)=0$ or $1$. We also say that $(X,\cF,\mu,T)$ is ergodic.
\end{definition}

\begin{definition}
Let $\cC$ be a family of certain subsets of a set $X$.
\begin{itemize}
\item[(1)] $\cC$ is called a \textit{monotone class} on $X$ if
  \begin{itemize}
  \item[\textcircled{\footnotesize{$1$}}] $A_n\in\cC$ and $A_1\subset A_2\subset\cdots$ $\Rightarrow$ $\bigcup_{n=1}^\infty A_n\in\cC$;
  \item[\textcircled{\footnotesize{$2$}}] $A_n\in\cC$ and $A_1\supset A_2\supset\cdots$ $\Rightarrow$ $\bigcap_{n=1}^\infty A_n\in\cC$.
  \end{itemize}
\item[(2)] $\cC$ is called a \textit{semi-algebra} on $X$ if
  \begin{itemize}
  \item[\textcircled{\footnotesize{$1$}}] $\emptyset\in\cC$;
  \item[\textcircled{\footnotesize{$2$}}] $A,B\in\cC\Rightarrow A\cap B\in\cC$;
  \item[\textcircled{\footnotesize{$3$}}] $A\in\cC\Rightarrow A^c\in\cC_{\Sigma f}$
  \end{itemize}
where $A^c:=X\setminus A$ and $\cC_{\Sigma f}:=\Big\{\bigcup_{i=1}^n C_i: C_1,\cdots,C_n\in\cC \text{ are disjoint, } n\in\N\Big\}$.
\newline (The subscript $_{\Sigma f}$ means finite disjoint union.)
\item[(3)] $\cC$ is called an \textit{algebra} on $X$ if
  \begin{itemize}
  \item[\textcircled{\footnotesize{$1$}}] $\emptyset,X\in\cC$;
  \item[\textcircled{\footnotesize{$2$}}] $A\in\cC\Rightarrow A^c\in\cC$;
  \item[\textcircled{\footnotesize{$3$}}] $A,B\in\cC\Rightarrow A\cap B\in\cC$.
  \end{itemize}
\item[(4)] $\cC$ is called a \textit{sigma-algebra} on $X$ if
  \begin{itemize}
  \item[\textcircled{\footnotesize{$1$}}] $\emptyset,X\in\cC$;
  \item[\textcircled{\footnotesize{$2$}}] $A\in\cC\Rightarrow A^c\in\cC$;
  \item[\textcircled{\footnotesize{$3$}}] $A_1,A_2,A_3\cdots\in\cC\Rightarrow \bigcap_{n=1}^\infty A_n\in\cC$.
  \end{itemize}
\end{itemize}
\end{definition}

The following useful approximation lemma follows from Theorem 0.1 and Theorem 0.7  in \cite{W82}.
\begin{lemma}\label{approximation}
Let $(X,\cB,\mu)$ be a probability space, $\cC$ be a semi-algebra which generates the sigma-algebra $\cB$ and $\cA$ be the algebra generated by $\cC$. Then
\begin{itemize}
\item[(1)] $\cA=\cC_{\Sigma f}:=\Big\{\bigcup_{i=1}^n C_i: C_1,\cdots,C_n\in\cC \text{ are disjoint, } n\in\N\Big\}$;
\item[(2)] for each $B\in\cB$ and each $\epsilon>0$, there is some $A\in\cA$ with $\mu(A\triangle B)<\epsilon$.
\end{itemize}
\end{lemma}

In order to extend some properties from a small family to a larger one in some proofs in Section 4, we recall the following well known Monotone Class Theorem (see for example \cite{F99}).

\begin{theorem}[Monotone Class Theorem]\label{monotone class theorem}
Let $\cA$ be an algebra. Then the smallest monotone class containing $\cA$ is precisely the smallest sigma-algebra containing $\cA$.
\end{theorem}

Besides, the following versions of some well known theorems (see for examples \cite{K14,W82,Y95}) will be used in Section 4.

\begin{theorem}[Carath\'eodory Measure Extension Theorem]
Let $\cC$ be a semi-algebra on $X$ and $\mu:\cC\to[0,+\infty]$ such that for all sets $A\in\cC$ for which there exists a countable decomposition $A=\cup_{i=1}^\infty A_i$ in disjoint sets $A_i\in\cC$ for $i\in\N$, we have $\mu(A)=\sum_{i=1}^\infty\mu(A_i)$. Then $\mu$ can be extended to become a measure $\mu'$ on $sig(\cC)$ (the smallest sigma-algebra containing $\cC$). That is, there exists a measure $\mu':sig(\cC)\to[0,+\infty]$ such that its restriction to $\cC$ is equal to $\mu$ (i.e., $\mu'|_{\cC}=\mu$). Moreover, if $X\in\cC$ and $\mu(X)<+\infty$, then the extension $\mu'$ is unique.
\end{theorem}

\begin{theorem}[Dominated Convergence Theorem]\label{convergence}
Let $(X,\cF,\mu)$ be a probability space and $\{f_n\}_{n\in\N}$ be a sequence of real-valued measurable functions on $X$ satisfying
$$\lim_{n\to\infty}f_n(x)=f(x)\quad\text{for }\mu\text{-almost every }x\in X.$$
If there exists a real-valued integrable function $g$ on $X$ such that for all $n\in\N$, $|f_n(x)|\le g(x)$ for $\mu$-almost every $x\in X$, then $f$ is integrable and
$$\lim_{n\to\infty}\int f_n d\mu=\int f d\mu.$$
\end{theorem}

\begin{theorem}[Vitali-Hahn-Saks Theorem]\label{measures limit}
Let $(X,\cF,\mu)$ be a probability space and $\{\lambda_n\}_{n\in\N}$ be a sequence of probability measures such that $\lambda_n\ll\mu$ for all $n\in\N$. If the finite $\lim_{n\to\infty}\lambda_n(B)=\lambda(B)$ exists for every $B\in\cF$, then $\lambda$ is countable additive on $\cF$.
\end{theorem}

\begin{theorem}[Birkhoff Ergodic Theorem]
Let $(X,\cF,\mu,T)$ be a measure-preserving dynamical system where the probability measure $\mu$ is ergodic with respect to $T$. Then for any real-valued integrable function $f:X\to\R$, we have
$$\lim_{n\to\infty}\frac{1}{n}\sum_{k=0}^{n-1}f(T^kx)=\int f d\mu$$
for $\mu$-almost every $x\in X$.
\end{theorem}

\section{Digit occurrence parameters}

The digit occurrence parameters and their properties studied in this section will be used in Section 4 and Section 6.

\begin{notation}[Digit occurrence parameters] Let $m\ge3$ be an integer. For any $w\in\Lambda_m^*$, write
$$\cN^m_0(w) := \Big\{k:1\le k\le |w|, w_k = 0 \text{ and } w_1\dots w_{k-1} 1\in\Lambda_m^*\Big\},$$
$$\cN^m_1(w) := \Big\{k:1\le k\le |w|, w_k = 1 \text{ and } w_1\dots w_{k-1} 0\in\Lambda_m^*\Big\},$$
and let
$$N^m_0(w) := \sharp \cN^m_0(w)\quad\text{and}\quad N^m_1(w) := \sharp \cN^m_1(w)$$
where $\sharp\cN$ denotes the cardinality of the set $\cN$.
\end{notation}

\begin{proposition}\label{parameter}
Let $m\ge3$ be an integer and $w,v \in \Lambda_m^*$ such that $wv \in \Lambda_m^*$. Then
\begin{itemize}
\item[\emph{(1)}] $N^m_0(w)+N^m_0(v)-1 \le N^m_0(wv) \le N^m_0(w) + N^m_0(v)$;
\item[\emph{(2)}] $N^m_1(w)+N^m_1(v)-1 \le N^m_1(wv) \le N^m_1(w) + N^m_1(v)$.
\end{itemize}
\end{proposition}
\begin{proof}
Let $a=|w|$ and $b=|v|$.
\item(1) \textcircled{\footnotesize{$1$}} Prove $N^m_0(wv) \le N^m_0(w) + N^m_0(v)$.
\newline It suffices to prove $\cN^m_0(wv)\subset\cN^m_0(w)\cup(\cN^m_0(v)+a)$, where $\cN_0^m(v)+a:=\{j+a:j\in\cN_0^m(v)\}$. Let $k\in\cN^m_0(wv)$.
\begin{itemize}
\item[i)] If $1\le k\le a$, then $w_k=0$, $w_1\cdots w_{k-1} 1\in\Lambda_m^*$ and we get $k\in\cN^m_0(w)$.
\item[ii)] If $a+1\le k\le a+b$, then $v_{k-a}=0$ and $w_1\cdots w_a v_1\cdots v_{k-a-1} 1\in\Lambda_m^*$. It follows from $v_1\cdots v_{k-a-1}1\in\Lambda_m^*$ that $k-a\in\cN^m_0(v)$ and $k\in\cN^m_0(v)+a$.
\end{itemize}
\textcircled{\footnotesize{$2$}} Prove $N^m_0(w)+N^m_0(v)\le N^m_0(wv)+1$.
\newline When $v=1^b$, we get $N^m_0(v)=0$ and then the conclusion follows from $N^m_0(w)\le N^m_0(wv)$.
\newline When $v\neq 1^b$, there exists a smallest $s\in\{1,\cdots,b\}$ such that $v_1=\cdots=v_{s-1}=1$ and $v_s=0$. In order to get the conclusion, it suffices to show $\cN^m_0(w)\cup(a+\cN^m_0(v))\subset\cN^m_0(wv)\cup\{a+s\}$. Since $\cN^m_0(w)\subset\cN^m_0(wv)$, we only need to prove $(a+\cN^m_0(v))\subset\cN^m_0(wv)\cup\{a+s\}$. Let $k\in\cN^m_0(v)\setminus\{s\}$. It suffices to check $a+k\in\cN^m_0(wv)$. By $v_k=0$, we only need to prove $w_1\cdots w_av_1\cdots v_{k-1}1\in\Lambda^*_m$. (By contradiction) Assume $w_1\cdots w_av_1\cdots v_{k-1}1\notin\Lambda_m^*$. Then $w_1\cdots w_av_1\cdots v_{k-1}1$ contains $0^m$ or $1^m$.
\begin{itemize}
\item[i)] If $w_1\cdots w_av_1\cdots v_{k-1}1$ contains $0^m$, then $w_1\cdots w_av_1\cdots v_{k-1}$ contains $0^m$. This contradicts $wv\in\Lambda_m^*$.
\item[ii)] If $w_1\cdots w_av_1\cdots v_{k-1}1$ contains $1^m$, by $k\ge s+1$, we know that
$$w_1\cdots w_av_1\cdots v_{s-1}0v_{s+1}\cdots v_{k-1}1$$
contains $1^m$. Thus $w_1\cdots w_av_1\cdots v_{s-1}$ contains $1^m$ or $v_{s+1}\cdots v_{k-1}1$ contains $1^m$. But $w_1\cdots w_av_1\cdots v_{s-1}$ contains $1^m$ will contradict $wv\in\Lambda^*_m$, and $v_{s+1}\cdots v_{k-1}1$ contains $1^m$ will imply $v_1\cdots v_{k-1}1$ contains $1^m$ which contradicts $k\in\cN^m_0(v)$.
\end{itemize}
(2) follows from the same way as (1).
\end{proof}

\begin{proposition}\label{lower bound of changable}
Let $m\ge3$ be an integer and $w\in\Lambda^*_m$. Then
\begin{itemize}
\item[\emph{(1)}] $m\cdot|w|_0\le(m-1)N^m_0(w)+|w|$;
\item[\emph{(2)}] $m\cdot|w|_1\le(m-1)N^m_1(w)+|w|$.
\end{itemize}
\end{proposition}
\begin{proof} (1) Let $n=|w|$. If $n\le m-1$, the conclusion follows immediately from $N^m_0(w)=|w|_0$. In the following, we assume $n\ge m$. Recall
$$\cN^m_0(w)=\Big\{k:1\le k\le n, w_k=0,w_1\cdots w_{k-1}1\in\Lambda^*_m\Big\}\quad\text{and}\quad N^m_0(w)=\sharp\cN^m_0(w).$$
We define
$$\cN^m_{1^{m-1}0}(w):=\Big\{k:m\le k\le n, w_{k-m+1}\cdots w_{k-1}w_k=1^{m-1}0\Big\}\quad\text{and}\quad N^m_{1^{m-1}0}=\sharp\cN^m_{1^{m-1}0}(w).$$
\begin{itemize}
\item[\textcircled{\footnotesize{$1$}}] Prove $\{k:1\le k\le n,w_k=0\}=\cN^m_0(w)\cup\cN^m_{1^{m-1}0}(w)$.
\newline$\boxed{\supset}$ Obvious.
\newline$\boxed{\subset}$ Let $k\in\{1,\cdots,n\}$ such that $w_k=0$. If $k\notin\cN^m_0(w)$, then $k\ge m$ and $w_1\cdots w_{k-1}1\notin\Lambda^*_m$. By $w_1\cdots w_{k-1}\in\Lambda^*_m$, we get $w_{k-m+1}\cdots w_{k-1}=1^{m-1}$. This implies $k\in\cN^m_{1^{m-1}0}(w)$.
\item[\textcircled{\footnotesize{$2$}}] Prove $\cN^m_0(w)\cap\cN^m_{1^{m-1}0}(w)=\emptyset$.
\newline(By contradiction) Assume that there exists $k\in\cN^m_0(w)\cap\cN^m_{1^{m-1}0}(w)$. Then $k\ge m$, $w_{k-m+1}\cdots w_{k-1}=1^{m-1}$ and $w_1\cdots w_{k-1}1\in\Lambda^*_m$. These imply $w_1\cdots w_{k-m}1^m\in\Lambda^*_m$, which contradicts the definition of $\Lambda^*_m$.
\end{itemize}
Combining \textcircled{\footnotesize{$1$}} and \textcircled{\footnotesize{$2$}}, we get $|w|_0=N^m_0(w)+N^m_{1^{m-1}0}(w)$. It follows from $(m-1)N^m_{1^{m-1}0}(w)\le|w|_1=|w|-|w|_0$ that $(m-1)(|w|_0-N^m_0(w))\le|w|-|w|_0$, i.e., $m\cdot|w|_0\le(m-1)N^m_0(w)+|w|$.
\newline(2) follows from the same way as (1).
\end{proof}

\section{Bernoulli-type measures on $\Lambda_m$}

Let $m\ge3$ be an integer, $\cB(\Lambda_m)$ be the Borel sigma-algebra on $\Lambda_m$ (equipped with the usual metric $d_2$) and $p\in(0,1)$. We define the \textit{$(p,1-p)$ Bernoulli-type measure} $\mu_p$ on $(\Lambda_m,\cB(\Lambda_m))$ as follows:
\begin{itemize}
\item[I.] Let
$$\mu_p(\emptyset)=0, \quad \mu_p(\Lambda_m)=1, \quad \mu_p [0] =p, \quad \text{and} \quad \mu_p[1]=1-p.$$
\item[II.] Suppose $\mu_p$ has been defined for all cylinders of order $n \in \N$. For any $w\in\Lambda_m^n$,
\newline if $w0, w1\in\Lambda_m^{n+1}$, we define
$$\mu_p[w0]:=p\mu_p[w] \quad \text{and} \quad \mu_p[w1]:=(1-p)\mu_p[w];$$
if $w0\in\Lambda_m^{n+1}$ but $w1\notin\Lambda_m^{n+1}$, then $[w1] = \emptyset$, $[w0]=[w]$ and naturally we have
$$\mu_p[w0]=\mu_p[w];$$
if $w1\in\Lambda_m^{n+1}$ but $w0\notin\Lambda_m^{n+1}$, then $[w0] = \emptyset$, $[w1]=[w]$ and naturally we have
$$\mu_p[w1]=\mu_p[w].$$
\item[III.] By Carath\'eodory's measure extension theorem, we uniquely extend $\mu_p$ from its definition on the family of cylinders to become a measure on $\cB(\Lambda_m)$.
\end{itemize}

\begin{remark}\label{measureformula} By the definition of $\mu_p$, we have
$$\mu_p[w] = p^{N^m_0(w)} (1-p)^{N^m_1(w)}\quad \text{for all } w \in \Lambda^*_m.$$
\end{remark}

Note that $\mu_p$ is not $\sigma_m$-invariant. In fact, $\mu_p[0^{m-2}1]=p^{m-2}(1-p)$ but
$$\mu_p(\sigma_m^{-1}[0^{m-2}1])=\mu_p[0^{m-1}1]+\mu_p[10^{m-2}1]=p^{m-1}+p^{m-2}(1-p)^2\neq p^{m-2}(1-p)\text{ for all }p\in(0,1).$$

The main result in this section is the following.

\begin{theorem}\label{lambda p}
Let $m\ge3$ be an integer and $p\in(0,1)$. Then there exists a unique $\sigma_m$-invariant ergodic probability measure $\lambda_p$ on $(\Lambda_m,\cB(\Lambda_m))$ equivalent to $\mu_p$, where $\lambda_p$ is defined by
$$\lambda_p(B):=\lim_{n\to\infty}\frac{1}{n}\sum_{k=0}^{n-1}\sigma_m^k\mu_p(B)\quad\text{for }B\in\cB(\Lambda_m).$$
\end{theorem}

The proof is based on the following lemmas.

\begin{lemma}\label{quasi-Bernoulli}
Let $m\ge3$ be an integer, $p\in(0,1)$ and $w,v\in\Lambda^*_m$ such that $wv\in\Lambda^*_m$. Then
$$\mu_p[w]\mu_p[v]\le\mu_p[wv]\le p^{-1}(1-p)^{-1}\mu_p[w]\mu_p[v].$$
\end{lemma}
\begin{proof}
It follows from Remark \ref{measureformula} and Proposition \ref{parameter}.
\end{proof}

\begin{lemma}\label{strong quasi-invariant}
Let $m\ge3$ be an integer and $p\in(0,1)$. Then there exists a constant $c>1$ such that
$$c^{-1}\mu_p(B)\le\sigma_m^k\mu_p(B)\le c\mu_p(B)$$
for all $k\in\N$ and $B\in\cB(\Lambda_m)$.
\end{lemma}
\begin{proof} Let $c=p^{-2}(1-p)^{-2}>1$.
\newline(1) Prove $c^{-1}\mu_p[w]\le\sigma_m^k\mu_p[w]\le c\mu_p[w]$ for any $k\in\N$ and $w\in\Lambda_m^*$.
\newline Fix $w\in\Lambda_m^*$ and $k\in\N$. Note that
$$\sigma_m^{-k}[w]=\bigcup_{u_1\cdots u_kw\in\Lambda_m^*}[u_1\cdots u_kw]$$
is a disjoint union.
\newline\textcircled{\footnotesize{$1$}} Estimate the upper bound of $\sigma_m^k\mu_p[w]$:
\begin{eqnarray*}
\mu_p\sigma_m^{-k}[w]&=&\sum_{u_1\cdots u_kw\in\Lambda_m^*}\mu_p[u_1\cdots u_kw]\\
&\overset{(\star)}{\le}&\sum_{u_1\cdots u_kw\in\Lambda_m^*}p^{-1}(1-p)^{-1}\mu_p[u_1\cdots u_k]\mu_p[w]\\
&\le& p^{-1}(1-p)^{-1}\sum_{u_1\cdots u_k\in\Lambda_m^*}\mu_p[u_1\cdots u_k]\mu_p[w]\\
&=&p^{-1}(1-p)^{-1}\mu_p[w]\\
&\le&c\mu_p[w]
\end{eqnarray*}
where $(\star)$ follows from Lemma \ref{quasi-Bernoulli}.
\newline\textcircled{\footnotesize{$2$}} Estimate the lower bound of $\sigma_m^k\mu_p[w]$:
\begin{itemize}
\item[i)] Prove $\mu_p\sigma_m^{-k}[0]\ge p^2(1-p)$ and $\mu_p\sigma_m^{-k}[1]\ge p(1-p)^2$. In fact, when $k=1$, the conclusion is obvious. When $k\ge2$, we have
\begin{eqnarray*}
\mu_p\sigma_m^{-k}[0]&=&\sum_{u_1\cdots u_k0\in\Lambda_m^*}\mu_p[u_1\cdots u_k0]\\
&\ge&\sum_{u_1\cdots u_{k-1}\overline{u}_{k-1}0\in\Lambda_m^*}\mu_p[u_1\cdots u_{k-1}\overline{u}_{k-1}0]\\
&\overset{(\star)}{=}&\sum_{u_1\cdots u_{k-1}\in\Lambda^*_m}\mu_p[u_1\cdots u_{k-1}\overline{u}_{k-1}0]\\
&\overset{(\star\star)}{\ge}&\mu_p[0]\sum_{u_1\cdots u_{k-1}\in\Lambda^*_m}\mu_p[u_1\cdots u_{k-1}]\mu_p[\overline{u}_{k-1}]\\
&\ge& p\sum_{u_1\cdots u_{k-1}\in\Lambda^*_m}\mu_p[u_1\cdots u_{k-1}]\cdot p(1-p)\\
&=&p^2(1-p),
\end{eqnarray*}
where $(\star)$ follows from
$$u_1\cdots u_{k-1}\overline{u}_{k-1}0\in\Lambda_m^*\Leftrightarrow u_1\cdots u_{k-1}\in\Lambda_m^*$$
and $(\star\star)$ follows from Lemma \ref{quasi-Bernoulli}. In the same way, we can get $\mu_p\sigma_m^{-k}[1]\ge p(1-p)^2$.
\item[ii)] Prove $\mu_p\sigma_m^{-k}[w]\ge c^{-1}\mu_p[w]$. In fact, when $w_1=0$, we have
\begin{eqnarray*}
\mu_p\sigma_m^{-k}[w]&=&\sum_{u_1\cdots u_kw\in\Lambda_m^*}\mu_p[u_1\cdots u_kw]\\
&\ge&\sum_{u_1\cdots u_{k-1}1w\in\Lambda_m^*}\mu_p[u_1\cdots u_{k-1}1w]\\
&\overset{(\star)}{=}&\sum_{u_1\cdots u_{k-1}1\in\Lambda_m^*}\mu_p[u_1\cdots u_{k-1}1w]\\
&\overset{(\star\star)}{\ge}&\sum_{u_1\cdots u_{k-1}1\in\Lambda_m^*}\mu_p[u_1\cdots u_{k-1}1]\mu_p[w]\\
&=&\mu_p\sigma_m^{-(k-1)}[1]\mu_p[w]\\
&\overset{(\star\star\star)}{\ge}&p(1-p)^2\mu_p[w].
\end{eqnarray*}
where $(\star)$ follows from $w_1=0$ and $w\in\Lambda^*_m$, $(\star\star)$ follows from Lemma \ref{quasi-Bernoulli} and $(\star\star\star)$ follows from i). When $w_1=1$, in the same way, we can get $\mu_p\sigma_m^{-k}[w]\ge p^2(1-p)\mu_p[w]$.
\end{itemize}
(2) Prove $c^{-1}\mu_p(B)\le\sigma_m^k\mu_p(B)\le c\mu_p(B)$ for all $k\in\N$ and $B\in\cB(\Lambda_m)$. Let $$\cC:=\Big\{[w]:w\in\Lambda_m^*\Big\}\cup\Big\{\emptyset\Big\},$$
$$\cC_{\Sigma f}:=\Big\{\bigcup_{i=1}^n C_i:C_1,\cdots,C_n\in\cC \text{ are disjoint, } n\in\N\Big\}$$
and
$$\cG:=\Big\{B\in\cB(\Lambda_m):c^{-1}\mu_p(B)\le\sigma_m^k\mu_p(B)\le c\mu_p(B) \text{ for all }k\in\N\Big\}.$$
Then $\cC$ is a semi-algebra on $\Lambda_m$, $\cC_{\Sigma f}$ is the algebra generated by $\cC$ (by Lemma \ref{approximation} (1)) and $\cG$ is a monotone class. Since in (1) we have already proved $\cC\subset\cG$, it follows that $\cC_{\Sigma f}\subset\cG\subset\cB(\Lambda_m)$. Noting that $\cB(\Lambda_m)$ is the smallest sigma-algebra containing $\cC_{\Sigma f}$, it follows from the Monotone Class Theorem (Theorem \ref{monotone class theorem}) that $\cG=\cB(\Lambda_m)$.
\end{proof}

\begin{lemma}[\cite{DM46}]\label{aeexist}
Let $(X,\cB,\mu)$ be a probability space and $T$ be a measurable transformation on $X$ satisfying $\mu(T^{-1}B)=0$ whenever $B\in\cB$ with $\mu(B)=0$. If there exists a constant $M>0$ such that for any $B\in\cB$ and any $n\in\N$,
$$\frac{1}{n}\sum_{k=0}^{n-1}\mu(T^{-k}B)\le M\mu(B),$$
then for any real integrable function $f$ on $X$, the limit
$$\lim_{n\to\infty}\frac{1}{n}\sum_{k=0}^{n-1}f(T^kx)$$
exists for $\mu$-almost every $x\in X$.
\end{lemma}

\begin{lemma}\label{01}
Let $m\ge3$ be an integer and $p\in(0,1)$. For any $B\in\cB(\Lambda_m)$ satisfying $\sigma_m^{-1}B=B$, we have $\mu_p(B)=0$ or $1$.
\end{lemma}
\begin{proof} Let $\alpha=p^2(1-p)^2>0$.
\item(1) Let $w\in\Lambda_m^*$ and $n=|w|$. For any $A\in\cB(\Lambda_m)$, we prove $\alpha\mu_p[w]\mu_p(A)\le\mu_p([w]\cap\sigma_m^{-(n+2)}A)$.
    \begin{itemize}
    \item[\textcircled{\footnotesize{$1$}}] For any $v\in\Lambda^*_m$, prove $\alpha\mu_p[w]\mu_p[v]\le\mu_p([w]\cap\sigma_m^{-(n+2)}[v])$.
        \newline In fact, it follows from $w\overline{w}_n\overline{v}_1v\in\Lambda_m^*$ and $[w]\cap\sigma_m^{-(n+2)}[v]\supset[w\overline{w}_n\overline{v}_1v]$ that
        $$\mu_p([w]\cap\sigma_m^{-(n+2)}[v])\ge\mu_p[w\overline{w}_n\overline{v}_1v]\overset{(\star)}{\ge}\mu_p[w]\mu_p[\overline{w}_n]\mu_p[\overline{v}_1]\mu_p[v]\ge(p(1-p))^2\mu_p[w]\mu_p[v]$$
        where $(\star)$ follows from Lemma \ref{quasi-Bernoulli}.
    \item[\textcircled{\footnotesize{$2$}}] Let
    $$\cC:=\Big\{[v]:v\in\Lambda_m^*\Big\}\cup\Big\{\emptyset\Big\}$$
    and
    $$\cG_w:=\Big\{A\in\cB(\Lambda_m):\alpha\mu_p[w]\mu_p(A)\le\mu_p([w]\cap\sigma_m^{-(n+2)}A)\Big\}.$$
    Then $\cG_w$ is a monotone class. Since in \textcircled{\footnotesize{$1$}} we have already proved $\cC\subset\cG_w$, in the same way as at the end of the proof of Lemma \ref{strong quasi-invariant}, we get $\cG_w=\cB(\Lambda_m)$.
    \end{itemize}
\item(2) We use $B^c$ to denote the complement of $B$ in $\Lambda_m$. For any $\epsilon>0$, by Lemma \ref{approximation}, there exist finitely many disjoint cylinders $\big\{[w^{(i)}]\big\}\subset\cC$ such that $\mu_p(B^c\Delta E_\epsilon)<\epsilon$ where $E_\epsilon=\bigcup_i[w^{(i)}]$.
\item(3) Let $B\in\cB(\Lambda_m)$ with $\sigma_m^{-1}B=B$. For any $w\in\Lambda_m^*$, by $B=\sigma_m^{-(|w|+2)}B$ and (1) we get
$$\alpha\mu_p(B)\mu_p[w]\le\mu(\sigma_m^{-(|w|+2)}B\cap[w])=\mu_p(B\cap[w]).$$
Thus
$$\alpha\mu_p(B)\mu_p(E_\epsilon)=\sum_i\alpha\mu_p(B)\mu_p[w^{(i)}]\le\sum_i\mu_p(B\cap[w^{(i)}])=\mu_p(B\cap\bigcup_i[w^{(i)}])=\mu_p(B\cap E_\epsilon).$$
Let $a=\mu_p((B\cup E_\epsilon)^c)$, $b=\mu_p(B\cap E_\epsilon)$, $c=\mu_p(B\setminus E_\epsilon)$ and $d=\mu_p(E_\epsilon\setminus B)$. Then we already have
$$\alpha(b+c)(b+d)\le b,\quad a+b<\epsilon\text{  (by }\mu_p(B^c\Delta E_\epsilon)<\epsilon\text{)}\quad\text{and}\quad a+b+c+d=1.$$
It follows from
$$\alpha(b+c)(a+d-\epsilon)\le\alpha(b+c)(b+d)\le b<\epsilon$$
that
$$(b+c)(a+d)<(\frac{1}{\alpha}+b+c)\epsilon\le(\frac{1}{\alpha}+1)\epsilon.$$
This implies $\mu_p(B)\mu_p(B^c)\le(\frac{1}{\alpha}+1)\epsilon$ for any $\epsilon>0$. Therefore $\mu_p(B)(1-\mu_p(B))=0$ and then $\mu_p(B)=0$ or $1$.
\end{proof}

\begin{proof}[Proof of Theorem \ref{lambda p}] (1) For any $n\in\N$ and $B\in\cB(\Lambda_m)$, define
$$\lambda_p^n(B):=\frac{1}{n}\sum_{k=0}^{n-1}\mu_p(\sigma_m^{-k}B).$$
Then $\lambda_p^n$ is a probability measure on $(\Lambda_m,\cB(\Lambda_m))$. By Lemma \ref{strong quasi-invariant}, there exists $c>0$ such that
\begin{align}\label{equi}c^{-1}\mu_p(B)\le \lambda_p^n(B)\le c\mu_p(B)\quad\text{for any }B\in\cB(\Lambda_m)\text{ and  }n\in\N.\end{align}
\item(2) For any $B\in\cB(\Lambda_m)$, prove that $\lim_{n\to\infty}\lambda_p^n(B)$ exists.
\newline Let $\mathbbm{1}_B:\Lambda_m\to\R$ be defined by
$$\mathbbm{1}_B(w):=\left\{\begin{array}{ll}
1 & \mbox{if } w\in B\\
0 & \mbox{if } w\notin B
\end{array}\right.$$
for any $w\in\Lambda_m$. Then
$$\begin{aligned}
\lim_{n\to\infty}\lambda_p^n(B)&=\lim_{n\to\infty}\frac{1}{n}\sum_{k=0}^{n-1}\int\mathbbm{1}_{\sigma_m^{-k}B}\text{ }d\mu_p\\
&=\lim_{n\to\infty}\int\frac{1}{n}\sum_{k=0}^{n-1}\mathbbm{1}_B(\sigma_m^kw)\text{ }d\mu_p(w)\\
&=\int\lim_{n\to\infty}\frac{1}{n}\sum_{k=0}^{n-1}\mathbbm{1}_B(\sigma_m^kw)\text{ }d\mu_p(w)
\end{aligned}$$
where the last equality follows from the Dominated Convergence Theorem (Theorem \ref{convergence}), in which the $\mu_p$-a.e. (almost every) existence of $\lim_{n\to\infty}\frac{1}{n}\sum_{k=0}^{n-1}\mathbbm{1}_B(\sigma_m^kw)$ follows from Lemma \ref{aeexist}, Lemma \ref{strong quasi-invariant} and (\ref{equi}).
\item(3) For any $B\in\cB(\Lambda_m)$, define
$$\lambda_p(B):=\lim_{n\to\infty}\lambda_p^n(B).$$
By the Vitali-Hahn-Saks Theorem (Theorem \ref{measures limit}), $\lambda_p$ is a probability measure on $(\Lambda_m,\cB(\Lambda_m))$.
\item(4) The fact $\lambda_p\sim\mu_p$ on $\cB(\Lambda_m)$ follows from \eqref{equi} and the definition of $\lambda_p$.
\item(5) Prove that $\lambda_p$ is $\sigma_m$-invariant.
\newline In fact, for any $B\in\cB(\Lambda_m)$ and $n\in\N$, we have
    $$\lambda_p^n(\sigma_m^{-1}B)=\frac{1}{n}\sum_{k=1}^n\mu_p(\sigma_m^{-k}B)=\frac{1}{n}\sum_{k=0}^n\mu_p(\sigma_m^{-k}B)-\frac{\mu_p(B)}{n}=\frac{n+1}{n}\lambda_p^{n+1}(B)-\frac{\mu_p(B)}{n}.$$
    Let $n\to\infty$, we get $\lambda_p(\sigma_m^{-1}B)=\lambda_p(B)$.
\item(6) Prove that $(\Lambda_m,\cB(\Lambda_m),\lambda_p,\sigma_m)$ is ergodic.
\newline In fact, for any $B\in\cB(\Lambda_m)$ satisfying $\sigma_m^{-1}B=B$, by Lemma \ref{01} we get $\mu_p(B)=0$ or $1$, which implies $\lambda_p(B)=0$ or $1$ since $\lambda_p\sim\mu_p$.
\item(7) Prove that such $\lambda_p$ is unique on $\cB(\Lambda_m)$.
\newline Let $\lambda_p'$ be a $\sigma_m$-invariant ergodic probability measure on $(\Lambda_m,\cB(\Lambda_m))$ equivalent to $\mu_p$. Then for any $B\in\cB(\Lambda_m)$, by the Birkhoff Ergodic Theorem, we get
$$\lambda_p'(B)=\int\mathbbm{1}_B\text{ }d\lambda_p'=\lim_{n\to\infty}\frac{1}{n}\sum_{k=0}^{n-1}\mathbbm{1}_B(\sigma_m^kw)\quad\text{for }\lambda_p'\text{-a.e. }w\in\Lambda_m$$
and
$$\lambda_p(B)=\int\mathbbm{1}_B\text{ }d\lambda_p=\lim_{n\to\infty}\frac{1}{n}\sum_{k=0}^{n-1}\mathbbm{1}_B(\sigma_m^kw)\quad\text{for }\lambda_p\text{-a.e. }w\in\Lambda_m.$$
Since $\lambda_p'\sim\mu_p\sim \lambda_p$, there exists $w\in\Lambda_m$ such that $\lambda_p'(B)=\lim_{n\to\infty}\frac{1}{n}\sum_{k=0}^{n-1}\mathbbm{1}_B(\sigma_m^k w)=\lambda_p(B)$. It means that $\lambda_p'$ and $\lambda_p$ are the same on $\cB(\Lambda_m)$.
\end{proof}

\section{Hausdorff dimension in metric space}

First we need the following concepts.

\begin{definition}[Hausdorff measure and dimension in metric space]\label{dim in metric space} Let $(X,d)$ be a metric space. For any $U\subset X$, denote the diameter of $U$ by $|U|:=\sup_{x,y\in U}d(x,y)$. For any $A\subset X, s\ge0$ and $\delta>0$, let
$$\cH^s_\delta(A,d):=\inf\Big\{\sum_{i=1}^\infty|U_i|^s:A\subset\bigcup_{i=1}^\infty U_i\text{ and }|U_i|\le\delta\text{ for all }i\in\N\Big\}.$$
The \textit{$s$-dimensional Hausdorff measure} of $A$ in $(X,d)$ is defined by
$$\cH^s(A,d):=\lim_{\delta\to0}\cH^s_\delta(A,d)$$
and the \textit{Hausdorff dimension} of $A$ in $(X,d)$ is defined by
$$\dim_H(A,d):=\sup\{s\ge0:\cH^s(A,d)=\infty\}.$$
In $\R^n$ (equipped with the usual metric), we use $\cH^s(A)$ and $\dim_H A$ to denote the $s$-dimensional Hausdorff measure and the Hausdorff dimension of $A$ respectively for simplification.
\end{definition}

\begin{definition}
Let $\mu$ be a finite Borel measure on a metric space $(X,d)$. The \textit{lower local dimension} of $\mu$ at $x\in X$ is defined by
$$\underline{\dim}_{loc}\mu(x):=\varliminf_{r\to 0}\frac{\log\mu(B(x,r))}{\log r},$$
where $B(x,r):=\{y\in X:d(y,x)\le r\}$ is the closed ball centered on $x$ with radius $r$.
\end{definition}

In $\R^n$, we can use the lower local dimension to estimate the upper and lower bounds of the Hausdorff dimension by the following proposition, which is also called Billingsley Lemma.

\begin{proposition}[\cite{F97} Proposition 2.3]\label{Billingsley in Rn}
Let $E\subset\R^n$ be a Borel set, $\mu$ be a finite Borel measure on $\R^n$ and $s\ge0$.
\begin{itemize}
\item[\emph{(1)}] If $\underline{\dim}_{loc}\mu(x)\le s$ for all $x\in E$, then $\dim_HE\le s$.
\item[\emph{(2)}] If $\underline{\dim}_{loc}\mu(x)\ge s$ for all $x\in E$ and $\mu(E)>0$, then $\dim_HE\ge s$.
\end{itemize}
\end{proposition}

We need to use the following version which is a generalization to metric spaces. For the sake of completeness we give a self-contained proof.

\begin{proposition}\label{Billingsley in metric}
Let $(X,d)$ be a metric space, $E\subset X$ be a Borel set, $\mu$ be a finite Borel measure on $X$ and $s\ge0$. If $\mu(E)>0$ and $\underline{\dim}_{loc}\mu(x)\ge s$ for all $x\in E$, then $\dim_H(E,d)\ge s$.
\end{proposition}
\begin{proof}
If $s=0$, the conclusion is obvious. If $s>0$, let $t\in(0,s)$. For any $x\in E$, by $\varliminf_{r\to0}\frac{\log\mu(B(x,r))}{\log r}>t$, there exists $\delta_x\in(0,1)$ such that for any $r\in(0,\delta_x)$, $\frac{\log\mu(B(x,r))}{\log r}>t$ and then $\mu(B(x,r))<r^t$. Thus $\varlimsup_{r\to0}\frac{\mu(B(x,r))}{r^t}\le1<2$ for all $x\in E$. By Lemma \ref{measure Billingsley in metric}, we get $\cH^t(E,d)\ge\frac{\mu(E)}{2}>0$. Thus $\dim_H(E,d)\ge t$ for all $t\in(0,s)$, which implies $\dim_H(E,d)\ge s$.
\end{proof}

\begin{lemma}\label{measure Billingsley in metric}
Let $(X,d)$ be a metric space, $E\subset X$ be a Borel set, $\mu$ be a finite Borel measure on $X$, $s\ge0$ and $c>0$. If $\varlimsup_{r\to0}\frac{\mu(B(x,r))}{r^s}<c$ for all $x\in E$, then $\cH^s(E,d)\ge\frac{\mu(E)}{c}$.
\end{lemma}
\begin{proof} For any $\delta>0$, let
$$E_\delta:=\{x\in E:\mu(B(x,r))\le cr^s\text{ for all }r\in(0,\delta)\}.$$
\item(1) Prove that $E_\delta$ is a Borel set. We define
$$F_q:=\{x\in E:\mu(B(x,q))\le cq^s\}\quad\text{for }q\in\Q.$$
It suffices to prove the following \textcircled{\footnotesize{$1$}} and \textcircled{\footnotesize{$2$}}.
\begin{itemize}
\item[\textcircled{\footnotesize{$1$}}] Prove $E_\delta=\bigcap_{q\in\Q\cap(0,\delta)}F_q$.
\newline$\boxed{\subset}$ follows from $E_\delta\subset F_q$ for all $q\in\Q\cap(0,\delta)$.
\newline$\boxed{\supset}$ Let $x\in\bigcap_{q\in\Q\cap(0,\delta)}F_q$. For any $r\in(0,\delta)$, there exist $q_1,q_2,\cdots,q_n,\cdots\in\Q\cap(0,\delta)$ decreasing to $r$. By $x\in\bigcap_{n=1}^\infty F_{q_n}$ we get $\mu(B(x,q_n))\le cq_n^s$ for all $n\in\N$. Thus
$$\mu(B(x,r))=\mu(\bigcap_{n=1}^\infty B(x,q_n))=\lim_{n\to\infty}\mu(B(x,q_n))\le\lim_{n\to\infty}cq_n^s=cr^s.$$
This implies $x\in E_\delta$.
\item[\textcircled{\footnotesize{$2$}}] Prove that $F_q$ is a Borel set.
\newline Define $f(x):=\mu(B(x,q))$ for $x\in X$. Then $F_q=E\cap f^{-1}(-\infty,cq^s]$. We only need to prove that $f$ is a Borel function. For any $a\in\R$, it suffices to prove that $f^{-1}(-\infty,a)$ is an open set. If $f^{-1}(-\infty,a)=\emptyset$, it is obviously open. We only need to consider $f^{-1}(-\infty,a)\neq\emptyset$ in the following.

Let $x_0\in f^{-1}(-\infty,a)$. Then $\mu(B(x_0,q))<a$. Since $\mu(B(x_0,q+\delta))$ decreases to $\mu(B(x_0,q))$ as $\delta$ decreases to $0$, there exists $\delta_0>0$ such that $\mu(B(x_0,q+\delta_0))<a$. It suffices to prove that the open ball $B^o(x_0,\delta_0):=\{x\in X:d(x,x_0)<\delta_0\}\subset f^{-1}(-\infty,a)$.

In fact, for any $x\in B^o(x_0,\delta_0)$, by $B(x,q)\subset B(x_0,q+\delta_0)$ we get $\mu(B(x,q))\le\mu(B(x_0,q+\delta_0))<a$, which implies $x\in f^{-1}(-\infty,a)$.
\end{itemize}
\item(2) Prove that $E_\delta$ increases to $E$ as $\delta$ decreases to $0$.
\begin{itemize}
\item[\textcircled{\footnotesize{$1$}}] If $0<\delta_2<\delta_1$, then obviously $E_{\delta_1}\subset E_{\delta_2}$.
\item[\textcircled{\footnotesize{$2$}}] Prove $E=\cup_{\delta>0}E_\delta$.
\newline$\boxed{\supset}$ follows from $E\supset E_\delta$ for all $\delta>0$.
\newline$\boxed{\subset}$ Let $x\in E$. By $\varlimsup_{r\to0}\frac{\mu(B(x,r))}{r^s}<c$, there exists $\delta_x>0$ such that for all $r\in(0,\delta_x)$, $\mu(B(x,r))\le cr^s$. Thus $x\in E_{\delta_x}\subset\cup_{\delta>0}E_\delta$.
\end{itemize}
\item(3) Prove $\cH^s(E,d)\ge\frac{\mu(E)}{c}$.
\newline Fix $\delta>0$. Let $\{U_k\}_{k\in K}$ be a countable $\delta$-cover of $E$, i.e.,
$$|U_k|\le\delta\text{ for all }k\in K\quad\text{and}\quad \bigcup_{k\in K}U_k\supset E\text{ }(\supset E_\delta).$$
Let $K':=\{k\in K: U_k\cap E_\delta\neq\emptyset\}$. Then $\bigcup_{k\in K'}U_k\supset E_\delta$. For any $k\in K'$, let $x_k\in U_k\cap E_\delta$ and $B_k:=B(x_k,|U_k|)\supset U_k$. Then $\bigcup_{k\in K'}B_k\supset E_\delta$. It follows that
$$\sum_{k\in K}|U_k|^s\ge\frac{1}{c}\sum_{k\in K'}c|U_k|^s\overset{(\star)}{\ge}\frac{1}{c}\sum_{k\in K'}\mu(B(x_k,|U_k|))\ge\frac{1}{c}\cdot\mu(\bigcup_{k\in K'}B_k)\ge\frac{\mu(E_\delta)}{c}$$
where $(\star)$ follows from $x_k\in E_\delta$. By the randomness of the choice of the $\delta$-cover $\{U_k\}_{k\in K}$, we get $\cH^s_\delta(E,d)\ge\frac{\mu(E_\delta)}{c}$ and then $\cH^s(E,d)\ge\frac{\mu(E_\delta)}{c}$. Let $\delta\to0$, by (1) and (2) we get $\cH^s(E,d)\ge\frac{\mu(E)}{c}$.
\end{proof}

\begin{definition}
Let $(X,d)$ and $(X',d')$ be two metric spaces. A map $f:X\to X'$ is called \textit{Lipschitz continuous} if there exists a constant $c>0$ such that
$$d'(f(x),f(y))\le c d(x,y)\quad\text{for all }x,y\in X.$$
\end{definition}

The following property of Lipschitz continuous map is well known (see for examples \cite{F90,F97}).

\begin{lemma}\label{Lipschitz}
Let $(X,d)$ and $(X',d')$ be two metric spaces and $f:X\to X'$ be a Lipschitz continuous map. Then for any $E\subset X$, we have
$$\dim_H(f(E),d')\le\dim_H(E,d).$$
\end{lemma}

In order to keep the paper self-contained we give a proof of the following useful proposition, which is a special case of \cite[Theorem 1.1]{L19}.

\begin{proposition}\label{dimension equality}
For any $E\subset\{0,1\}^\N$, we have
$$\dim_H(E,d_2)=\dim_H \pi_2(E).$$
\end{proposition}
\begin{proof} $\boxed{\ge}$ follows from the fact that $\pi_2:\{0,1\}^\N\to[0,1]$ is Lipschitz continuous.
\newline$\boxed{\le}$ (1) Prove that for any $s>0$, $\cH^s(E,d_2)\le2^{s+1}\cH^s(\pi_2(E))$.
\newline Fix $\delta\in(0,\frac{1}{2})$. Let $\{U_k\}_{k\in K}$ be a countable $\delta$-cover of $\pi_2(E)$, i.e., $0<|U_k|\le\delta$ and $\pi_2(E)\subset\bigcup_{k\in K} U_k$. Then for each $U_k$, there exists $n_k\in\N$ such that $\frac{1}{2^{n_k+1}}<|U_k|\le\frac{1}{2^{n_k}}$. Since the length of every cylinder in $[0,1]$ of order $n_k$ is $\frac{1}{2^{n_k}}$, every $U_k$ can be cover by at most two consecutive cylinders $I_{k,1}$ and $I_{k,2}$ of order $n_k$. This implies
$$\pi_2(E)\subset\bigcup_{k\in K}\bigcup_{j=1,2}I_{k,j}\quad\text{and then}\quad E\subset\bigcup_{k\in K}\bigcup_{j=1,2}\pi_2^{-1}I_{k,j}.$$
By $|\pi_2^{-1}I_{k,j}|=\frac{1}{2^{n_k}}<2|U_k|\le2\delta$, we get
$$\cH^s_{2\delta}(E,d_2)\le\sum_{k\in K}\sum_{j=1,2}|\pi_2^{-1}I_{k,j}|^s=\sum_{k\in K}\frac{2}{2^{n_ks}}=2^{s+1}\sum_{k\in K}(\frac{1}{2^{n_k+1}})^s<2^{s+1}\sum_{k\in K}|U_k|^s.$$
It follows from the randomness of the choice of the $\delta$-cover $\{U_k\}_{k\in K}$ that $\cH^s_{2\delta}(E,d_2)\le2^{s+1}\cH^s_\delta(\pi_2(E))$. Let $\delta\to 0$, we get $\cH^s(E,d_2)\le2^{s+1}\cH^s(\pi_2(E))$.
\item(2) Prove $\dim_H(E,d_2)\le\dim_H\pi_2(E)$.
\newline If $\dim_H(E,d_2)=0$, the conclusion is obvious. If $\dim_H(E,d_2)>0$, then for any $s\in(0,\dim_H(E,d_2))$, by $\cH^s(E,d_2)=\infty$ and (1), we get $\cH^s(\pi_2(E))=\infty$. It means that $s\le\dim_H\pi_2(E)$ for any $s\in(0,\dim_H(E,d_2))$. Thus $\dim_H(E,d_2)\le\dim_H\pi_2(E)$.
\end{proof}

\section{Proof of the main result}

For any $x\in[0,1)$, let $\epsilon_1(x)\epsilon_2(x)\cdots\epsilon_n(x)\cdots$ be the \textit{greedy binary expansion} of $x$. Recall the following well known result.

\begin{theorem}[\cite{E49}]\label{classical} For any $p\in[0,1]$, we have
$$\dim_H\Big\{x\in[0,1):\lim_{n\to\infty}\frac{|\epsilon_1(x)\cdots \epsilon_n(x)|_0}{n}=p\Big\}=\frac{-p\log p-(1-p)\log(1-p)}{\log2}$$
where $0\log0:=0$.
\end{theorem}

The following lemma is a special case of \cite[Theorem 6.1]{LLS19}.

\begin{lemma}\label{upper bound special of LLS}
For any $p\in[0,1]$, we have
$$\dim_H\Big\{x\in[0,1):\varliminf_{n\to\infty}\frac{|\epsilon_1(x)\cdots \epsilon_n(x)|_0}{n}=p\Big\}\le\frac{-p\log p-(1-p)\log(1-p)}{\log2}$$
and
$$\dim_H\Big\{x\in[0,1):\varlimsup_{n\to\infty}\frac{|\epsilon_1(x)\cdots \epsilon_n(x)|_0}{n}=p\Big\}\le\frac{-p\log p-(1-p)\log(1-p)}{\log2}$$
where $0\log0:=0$.
\end{lemma}

For any $p\in[0,1]$, we define the \textit{global frequency sets} in $\{0,1\}^\N$ by
$$G_p:=\Big\{w\in\{0,1\}^\N:\lim_{n\to\infty}\frac{|w_1\cdots w_n|_0}{n}=p\Big\},$$
$$\underline{G}_p:=\Big\{w\in\{0,1\}^\N:\varliminf_{n\to\infty}\frac{|w_1\cdots w_n|_0}{n}=p\Big\}$$
and
$$\overline{G}_p:=\Big\{w\in\{0,1\}^\N:\varlimsup_{n\to\infty}\frac{|w_1\cdots w_n|_0}{n}=p\Big\}.$$

By Theorem \ref{classical}, Lemma \ref{upper bound special of LLS} and Proposition \ref{dimension equality}, the following is obvious.

\begin{theorem}\label{global dimension} For any $p\in[0,1]$, we have
$$\dim_H (G_p,d_2)=\dim_H (\underline{G}_p,d_2)=\dim_H (\overline{G}_p,d_2)=\frac{-p\log p-(1-p)\log(1-p)}{\log2}.$$
\end{theorem}

To prove Theorem \ref{main}, we need the next lemma, which will be proved later. For any integer $m\ge3$, we recall
$$\Lambda_m:=\Big\{w\in\{0,1\}^\N: w \text{ does not contain } 0^m \text{ or } 1^m\Big\}$$
and define
$$\Lambda_{m,p}:=\Lambda_m\cap G_p\quad\text{for }p\in[0,1].$$

\begin{lemma}\label{lowerbound}
Let $p\in(0,1)$ and integer $m\ge3$ be large enough such that $\frac{1}{m}<p<1-\frac{1}{m}$. Let $f_m:(0,1)\to\R$ be defined by
$$f_m(x):=\frac{x-x^m}{1-x^m-(1-x)^m}\quad\text{for }x\in(0,1).$$
Then there exists $q_m\in(0,1)$ such that $f_m(q_m)=p$ and
$$\dim_H(\Lambda_{m,p},d_2)\ge\frac{-(mp-1)\log q_m-(m-mp-1)\log(1-q_m)}{(m-1)\log2}.$$
Moreover, $q_m\to p$ as $m\to\infty$.
\end{lemma}

\begin{proof}[Proof of Theorem \ref{main}] First we prove (2). Let $p\in[0,1]$. By Proposition \ref{dimension equality}, it suffices to prove
$$\begin{aligned}
\dim_H(\Gamma_p,d_2)=\dim_H(\underline{\Gamma}_p,d_2)=\dim_H(\overline{\Gamma}_p,d_2)&\\
=\dim_H(\Lambda_p,d_2)=\dim_H(\underline{\Lambda}_p,d_2)=\dim_H(\overline{\Lambda}_p,d_2)&=\frac{-p\log p-(1-p)\log(1-p)}{\log 2}.
\end{aligned}$$
Since it is straightforward to check $\Gamma\subset\Lambda$, we have
$$\Gamma_p\subset\Lambda_p\subset G_p,\quad\Gamma_p\subset\underline{\Gamma}_p\subset\underline{\Lambda}_p\subset\underline{G}_p\quad\text{and}\quad\Gamma_p\subset\overline{\Gamma}_p\subset\overline{\Lambda}_p\subset\overline{G}_p.$$
By Theorem \ref{global dimension}, we only need to prove
\begin{eqnarray}\label{Gamma_p lowerbound}
\dim_H(\Gamma_p,d_2)\ge\frac{-p\log p-(1-p)\log(1-p)}{\log2}.
\end{eqnarray}
If $p=0$ or $1$, this follows immediately from $0\log0:=0$ and $1\log1=0$. So we only need to consider $0<p<1$ in the following. For any integer $m\ge3$, we define
$$\Theta_{m,p}:=\Big\{w\in G_p:w_1\cdots w_{2m}=1^{2m},w_{km+1}\cdots w_{km+m}\notin\{0^m,1^m\}\text{ for all }k\ge2\Big\}$$
and
$$\Xi_{m,p}:=\Big\{w\in G_p:w_{km+1}\cdots w_{km+m}\notin\{0^m,1^m\}\text{ for all }k\ge0\Big\}.$$
Then
\begin{eqnarray}\label{string}
\dim_H(\Gamma_p,d_2)\overset{(\star)}{\ge}\dim_H(\Theta_{m,p},d_2)\overset{(\star\star)}{\ge}\dim_H(\Xi_{m,p},d_2)\overset{(\star\star\star)}{\ge}\dim_H(\Lambda_{m,p},d_2)
\end{eqnarray}
where $(\star)$ follows from $\Gamma_p\supset\Theta_{m,p}$, $(\star\star\star)$ follows from $\Xi_{m,p}\supset\Lambda_{m,p}$ and $(\star\star)$ follows from Lemma \ref{Lipschitz}, $\sigma^{2m}(\Theta_{m,p})=\Xi_{m,p}$ and the fact that $\sigma^{2m}$ is Lipschitz continuous (since $d_2(\sigma^{2m}(w),\sigma^{2m}(v))\le2^{2m}d_2(w,v)$ for all $w,v\in\{0,1\}^\N$). By (\ref{string}) and Lemma \ref{lowerbound}, for $m$ large enough, there exists $q_m\in(0,1)$ such that $q_m\to p$ (as $m\to\infty$) and
$$\dim_H(\Gamma_p,d_2)\ge\frac{-(mp-1)\log q_m-(m-mp-1)\log(1-q_m)}{(m-1)\log2}.$$
Let $m\to\infty$, we get (\ref{Gamma_p lowerbound}).

Finally we deduce (1) from (2). In fact, since (2) implies $$\dim_H\pi_2(\Gamma_{\frac{1}{2}})=\dim_H(\Gamma_{\frac{1}{2}},d_2)=1,$$
it follows from $\Gamma_{\frac{1}{2}}\subset\Gamma\subset\Lambda\subset\{0,1\}^\N$ that
$$\dim_H\pi_2(\Gamma)=\dim_H(\Gamma,d_2)=\dim_H\pi_2(\Lambda)=\dim_H(\Lambda,d_2)=1.$$
\end{proof}

\begin{proof}[Proof of Lemma \ref{lowerbound}]
Since $f_m$ is continuous on $(0,1)$, $\lim_{x\to0^+}f_m(x)=\frac{1}{m}$, $\lim_{x\to1^-}f_m(x)=1-\frac{1}{m}$ and $\frac{1}{m}<p<1-\frac{1}{m}$, there exists $q_m\in(0,1)$ such that $f_m(q_m)=p$.

\item(1) Prove $q_m\to p$ as $m\to\infty$. Notice that
$$|q_m-p|=|q_m-f_m(q_m)|=\Big|\frac{q_m^m(1-q_m)-q_m(1-q_m)^m}{1-q_m^m-(1-q_m)^m}\Big|.$$
Let
$$g_m(x):=\frac{x^m(1-x)-x(1-x)^m}{1-x^m-(1-x)^m}\quad\text{for }x\in(0,1).$$
Then
$$|q_m-p|=|g_m(q_m)|\le\sup_{x\in(0,1)}|g_m(x)|.$$
In order to prove $q_m\to p$, it suffices to check $|g_m(x)|\le\frac{1}{m}$ for all $x\in(0,1)$. That is,
$$m\cdot|x^m(1-x)-x(1-x)^m|\le1-x^m-(1-x)^m\quad\text{for all }x\in(0,1).$$
\begin{itemize}
\item[\textcircled{\footnotesize{$1$}}] When $x\in(0,\frac{1}{2}]$, we get $x^m(1-x)-x(1-x)^m\le0$. It suffices to prove $(m-mx-1)x^m+1-(mx+1)(1-x)^m\ge0$. Since $m-mx-1>0$, we only need to prove $h_m(x):=(mx+1)(1-x)^m\le1$ for all $x\in[0,\frac{1}{2}]$. This follows from $h_m(0)=1$ and $h_m'(x)=-m(m+1)x(1-x)^{m-1}\le0$ for all $x\in[0,\frac{1}{2}]$.
\item[\textcircled{\footnotesize{$2$}}] When $x\in(\frac{1}{2},1)$, we get $x^m(1-x)-x(1-x)^m\ge0$. It suffices to prove $(mx-1)(1-x)^m+1-(1+m-mx)x^m\ge0$. Since $mx-1>0$, we only need to prove $h_m(x):=(1+m-mx)x^m\le1$ for all $x\in[\frac{1}{2},1]$. This follows from $h_m(1)=1$ and $h_m'(x)=m(m+1)(1-x)x^{m-1}\ge0$ for all $x\in[\frac{1}{2},1]$.
\end{itemize}
\item(2) We apply Proposition \ref{Billingsley in metric} to get the lower bound of $\dim_H(\Lambda_{m,p},d_2)$. Let $\mu_{q_m}$ be the $(q_m,1-q_m)$ Bernoulli-type measure on $(\Lambda_m,\cB(\Lambda_m))$ defined in Section 4.
\begin{itemize}
\item[\textcircled{\footnotesize{$1$}}] Prove that $\Lambda_{m,p}:=\Lambda_m\cap G_p$ is a Borel set in $(\Lambda_m,d_2)$.
    \newline In fact, it suffices to prove that $G_p$ is a Borel set in $(\{0,1\}^\N,d_2)$. Since
    $$\frac{|w_1\cdots w_n|_0}{n}=\frac{1}{n}\sum_{k=0}^{n-1}\mathbbm{1}_{[0]}(\sigma^k(w))$$
    where $\sigma$ is a Borel map and $\mathbbm{1}_{[0]}$ is a Borel function on $(\{0,1\}^\N,d_2)$, we know that $\underline{G}_p$ and $\overline{G}_p$ are Borel sets, which implies that $G_p=\underline{G}_p\cap\overline{G}_p$ is a Borel set.
\item[\textcircled{\footnotesize{$2$}}] Prove $\mu_{q_m}(\Lambda_{m,p})=1$.
\newline Let $\lambda_{q_m}$ be the measure defined in Theorem \ref{lambda p} such that $(\Lambda_m,\cB(\Lambda_m),\lambda_{q_m},\sigma_m)$ is ergodic. It follows from the Birkhoff Ergodic Theorem that
$$\lim_{n\to\infty}\frac{1}{n}\sum_{k=0}^{n-1}\mathbbm{1}_{[0]}(\sigma_m^{k}w)=\int\mathbbm{1}_{[0]}d\lambda_{q_m}=\lambda_{q_m}[0]\xlongequal[\text{Lemma \ref{lambda_p[0]}}]{\text{by}}\frac{q_m-q_m^m}{1-q_m^m-(1-q_m)^m}=f_m(q_m)=p$$
for $\lambda_{q_m}$-almost every $w\in\Lambda_m$. By $\frac{|w_1\cdots w_n|_0}{n}=\frac{1}{n}\sum_{k=0}^{n-1}\mathbbm{1}_{[0]}(\sigma_m^{k}w)$, we get
$$\lim_{n\to\infty}\frac{|w_1\cdots w_n|_0}{n}=p\quad\text{for }\lambda_{q_m}\text{-almost every }w\in\Lambda_m,$$
which implies $\lambda_{q_m}(\Lambda_{m,p})=1$. It follows from $\lambda_{q_m}\sim\mu_{q_m}$ that $\mu_{q_m}(\Lambda_{m,p})=1$.
\item[\textcircled{\footnotesize{$3$}}] For any $w\in\Lambda_{m,p}$, we have
\begin{small}
$$\begin{aligned}
\underline{\dim}_{loc}\mu_{q_m}(w)&=\varliminf_{r\to\infty}\frac{\log\mu_{q_m}(B(w,r))}{\log r}\\
&\overset{(\star)}{\ge}\varliminf_{n\to\infty}\frac{\log\mu_{q_m}[w_1\cdots w_n]}{\log2^{-n}}\\
&=\varliminf_{n\to\infty}\frac{-\log q_m^{N^m_0(w_1\cdots w_n)}(1-q_m)^{N^m_1(w_1\cdots w_n)}}{n\log2}\\
&\ge\frac{\varliminf_{n\to\infty}\frac{N^m_0(w_1\cdots w_n)}{n}(-\log q_m)+\varliminf_{n\to\infty}\frac{N^m_1(w_1\cdots w_n)}{n}(-\log(1-q_m))}{\log2}\\
&\overset{(\star\star)}{\ge}\frac{\varliminf_{n\to\infty}\big(\frac{m\cdot|w_1\cdots w_n|_0}{(m-1)n}-\frac{1}{m-1}\big)(-\log q_m)+\varliminf_{n\to\infty}\big(\frac{m\cdot|w_1\cdots w_n|_1}{(m-1)n}-\frac{1}{m-1}\big)(-\log(1-q_m))}{\log2}\\
&\overset{(\star\star\star)}{=}\frac{-(mp-1)\log q_m-(m-mp-1)\log(1-q_m)}{(m-1)\log2}
\end{aligned}$$
\end{small}where $(\star\star\star)$ follows from $w\in\Lambda_{m,p}$, $(\star\star)$ follows from Proposition \ref{lower bound of changable} and $(\star)$ can be proved as follows. For any $r\in(0,1)$, there exists $n=n(r)\in\N$ such that $\frac{1}{2^n}\le r<\frac{1}{2^{n-1}}$. Then by $B(w,r)=[w_1\cdots w_n]$ and $\log\mu_{q_m}[w_1\cdots w_n]<0$, we get $\frac{\log\mu_{q_m}(B(w,r))}{\log r}\ge\frac{\log\mu_{q_m}[w_1\cdots w_n]}{\log2^{-n}}$. (In fact, $(\star)$ can take ``$=$''.)
\end{itemize}
Thus the lower bound of $\dim_H(\Lambda_{m,p},d_2)$ follows from \textcircled{\footnotesize{$1$}}, \textcircled{\footnotesize{$2$}}, \textcircled{\footnotesize{$3$}} and Proposition \ref{Billingsley in metric}.
\end{proof}

To end this paper, we prove the following lemma, which has been used in the above proof.

\begin{lemma}\label{lambda_p[0]} Let $m\ge3$ be an integer, $p\in(0,1)$ and $\lambda_p$ be the measure on $(\Lambda_m,\cB(\Lambda_m))$ defined in Theorem \ref{lambda p}. Then
$$\lambda_p[0]=\frac{p-p^m}{1-p^m-(1-p)^m}.$$
\end{lemma}
\begin{proof} By the definition of $\lambda_p$, we know
$$\lambda_p[0]=\lim_{n\to\infty}\frac{1}{n}\sum_{k=0}^{n-1}\mu_p\sigma_m^{-k}[0].$$
For any integer $k\ge0$, let
$$a_k:=\mu_p\sigma_m^{-k}[0]=\sum_{u_1\cdots u_k0\in\Lambda^*_m}\mu_p[u_1\cdots u_k0],\quad b_k:=\mu_p\sigma_m^{-k}[1]=\sum_{u_1\cdots u_k1\in\Lambda^*_m}\mu_p[u_1\cdots u_k1],$$
$$c_k:=\mu_p\sigma_m^{-k}[01]=\sum_{u_1\cdots u_k01\in\Lambda^*_m}\mu_p[u_1\cdots u_k01],\quad d_k:=\mu_p\sigma_m^{-k}[10]=\sum_{u_1\cdots u_k10\in\Lambda^*_m}\mu_p[u_1\cdots u_k10].$$
By Theorem \ref{lambda p}, the following limits exist:
$$a:=\lim_{n\to\infty}\frac{1}{n}\sum_{k=0}^{n-1}a_k=\lambda_p[0],\quad b:=\lim_{n\to\infty}\frac{1}{n}\sum_{k=0}^{n-1}b_k=\lambda_p[1],$$
$$c:=\lim_{n\to\infty}\frac{1}{n}\sum_{k=0}^{n-1}c_k=\lambda_p[01],\quad
d:=\lim_{n\to\infty}\frac{1}{n}\sum_{k=0}^{n-1}d_k=\lambda_p[10].$$
(1) We have $a+b=1$ since $\lambda_p[0]+\lambda_p[1]=\lambda_p(\Lambda_m)$.
\newline(2) We have $c=d$ since $\lambda_p[00]+\lambda_p[01]=\lambda_p[0]=\lambda_p\sigma_m^{-1}[0]=\lambda_p[00]+\lambda_p[10]$.
\newline(3) Prove $(1-p)a+p^{m-1}d=c$ and $pb+(1-p)^{m-1}c=d$.
\newline\textcircled{\footnotesize{$1$}} For $k\ge m$, we have
$$a_k=d_{k-1}+pd_{k-2}+\cdots+p^{m-3}d_{k-m+2}+p^{m-2}d_{k-m+1},$$
since
\begin{small}
$$\begin{aligned}
&\indent\sum_{u_1\cdots u_k0\in\Lambda^*_m}\mu_p[u_1\cdots u_k0]\\
&=\sum_{u_1\cdots u_{k-1}10\in\Lambda^*_m}\mu_p[u_1\cdots u_{k-1}10]+\sum_{u_1\cdots u_{k-1}00\in\Lambda^*_m}\mu_p[u_1\cdots u_{k-1}00]\\
&=d_{k-1}+\sum_{u_1\cdots u_{k-2}100\in\Lambda^*_m}\mu_p[u_1\cdots u_{k-2}100]+\sum_{u_1\cdots u_{k-2}000\in\Lambda^*_m}\mu_p[u_1\cdots u_{k-2}000]\\
&\overset{(\star)}{=}d_{k-1}+\sum_{u_1\cdots u_{k-2}10\in\Lambda^*_m}p\mu_p[u_1\cdots u_{k-2}10]+\sum_{u_1\cdots u_{k-2}000\in\Lambda^*_m}\mu_p[u_1\cdots u_{k-2}000]\\
&=d_{k-1}+pd_{k-2}+\sum_{u_1\cdots u_{k-3}1000\in\Lambda^*_m}\mu_p[u_1\cdots u_{k-3}1000]+\sum_{u_1\cdots u_{k-3}0000\in\Lambda^*_m}\mu_p[u_1\cdots u_{k-3}0000]\\
&\overset{(\star\star)}{=}d_{k-1}+pd_{k-2}+\sum_{u_1\cdots u_{k-3}10\in\Lambda^*_m}p^2\mu_p[u_1\cdots u_{k-3}10]+\sum_{u_1\cdots u_{k-3}0^4\in\Lambda^*_m}\mu_p[u_1\cdots u_{k-3}0^4]\\
&=\cdots\\
&=d_{k-1}+pd_{k-2}+\cdots+p^{m-3}d_{k-m+2}+\sum_{u_1\cdots u_{k-m+2}0^{m-1}\in\Lambda^*_m}\mu_p[u_1\cdots u_{k-m+2}0^{m-1}]\\
&=d_{k-1}+pd_{k-2}+\cdots+p^{m-3}d_{k-m+2}+\sum_{u_1\cdots u_{k-m+1}10^{m-1}\in\Lambda^*_m}\mu_p[u_1\cdots u_{k-m+1}10^{m-1}]\\
&\overset{(\star\star\star)}{=}d_{k-1}+pd_{k-2}+\cdots+p^{m-3}d_{k-m+2}+\sum_{u_1\cdots u_{k-m+1}10\in\Lambda^*_m}p^{m-2}\mu_p[u_1\cdots u_{k-m+1}10]\\
&=d_{k-1}+pd_{k-2}+\cdots+p^{m-3}d_{k-m+2}+p^{m-2}d_{k-m+1},
\end{aligned}$$
\end{small}
where $(\star)$, $(\star\star)$ and $(\star\star\star)$ follow from
$$\begin{aligned}
&u_1\cdots u_{k-2}100\in\Lambda^*_m\Leftrightarrow u_1\cdots u_{k-2}10\in\Lambda^*_m\\
&\Rightarrow u_1\cdots u_{k-2}101\in\Lambda^*_m,
\end{aligned}$$
$$\begin{aligned}
&u_1\cdots u_{k-3}1000\in\Lambda^*_m\Leftrightarrow u_1\cdots u_{k-3}10\in\Lambda^*_m\\
&\Rightarrow u_1\cdots u_{k-3}101, u_1\cdots u_{k-3}1001\in\Lambda^*_m
\end{aligned}$$
and
$$\begin{aligned}
&u_1\cdots u_{k-m+1}10^{m-1}\in\Lambda^*_m\Leftrightarrow u_1\cdots u_{k-m+1}10\in\Lambda^*_m\\
&\Rightarrow u_1\cdots u_{k-m+1}101, u_1\cdots u_{k-m+1}1001, \cdots, u_1\cdots u_{k-m+1}10^{m-2}1\in\Lambda^*_m
\end{aligned}$$
respectively, recalling the definition of $\mu_p$.
\newline\textcircled{\footnotesize{$2$}} For $k\ge m$, we have
$$c_k=(1-p)d_{k-1}+(1-p)pd_{k-2}+\cdots+(1-p)p^{m-3}d_{k-m+2}+p^{m-2}d_{k-m+1},$$
since
\begin{small}
$$\begin{aligned}
&\indent\sum_{u_1\cdots u_k01\in\Lambda^*_m}\mu_p[u_1\cdots u_k01]\\
&=\sum_{u_1\cdots u_{k-1}101\in\Lambda^*_m}\mu_p[u_1\cdots u_{k-1}101]+\sum_{u_1\cdots u_{k-1}001\in\Lambda^*_m}\mu_p[u_1\cdots u_{k-1}001]\\
&\overset{(\star)}{=}\sum_{u_1\cdots u_{k-1}10\in\Lambda^*_m}(1-p)\mu_p[u_1\cdots u_{k-1}10]+\sum_{u_1\cdots u_{k-1}001\in\Lambda^*_m}\mu_p[u_1\cdots u_{k-1}001]\\
&=(1-p)d_{k-1}+\sum_{u_1\cdots u_{k-2}1001\in\Lambda^*_m}\mu_p[u_1\cdots u_{k-2}1001]+\sum_{u_1\cdots u_{k-2}0001\in\Lambda^*_m}\mu_p[u_1\cdots u_{k-2}0001]\\
&\overset{(\star\star)}{=}(1-p)d_{k-1}+\sum_{u_1\cdots u_{k-2}10\in\Lambda^*_m}p(1-p)\mu_p[u_1\cdots u_{k-2}10]+\sum_{u_1\cdots u_{k-2}0^31\in\Lambda^*_m}\mu_p[u_1\cdots u_{k-2}0^31]\\
&=(1-p)d_{k-1}+p(1-p)d_{k-2}+\sum_{u_1\cdots u_{k-3}10^31\in\Lambda^*_m}\mu_p[u_1\cdots u_{k-3}10^31]+\sum_{u_1\cdots u_{k-3}0^41\in\Lambda^*_m}\mu_p[u_1\cdots u_{k-3}0^41]\\
&=\cdots\\
&=(1-p)d_{k-1}+(1-p)pd_{k-2}+\cdots+(1-p)p^{m-3}d_{k-m+2}+\sum_{u_1\cdots u_{k-m+2}0^{m-1}1\in\Lambda^*_m}\mu_p[u_1\cdots u_{k-m+2}0^{m-1}1]\\
&=(1-p)d_{k-1}+(1-p)pd_{k-2}+\cdots+(1-p)p^{m-3}d_{k-m+2}+\sum_{u_1\cdots u_{k-m+1}10^{m-1}1\in\Lambda^*_m}\mu_p[u_1\cdots u_{k-m+1}10^{m-1}1]\\
&\overset{(\star\star\star)}{=}(1-p)d_{k-1}+(1-p)pd_{k-2}+\cdots+(1-p)p^{m-3}d_{k-m+2}+\sum_{u_1\cdots u_{k-m+1}10\in\Lambda^*_m}p^{m-2}\mu_p[u_1\cdots u_{k-m+1}10]\\
&=(1-p)d_{k-1}+(1-p)pd_{k-2}+\cdots+(1-p)p^{m-3}d_{k-m+2}+p^{m-2}d_{k-m+1},
\end{aligned}$$
\end{small}
where $(\star)$, $(\star\star)$ and $(\star\star\star)$ follow from
$$\begin{aligned}
&u_1\cdots u_{k-1}101\in\Lambda^*_m\Leftrightarrow u_1\cdots u_{k-1}10\in\Lambda^*_m\\
&\Rightarrow u_1\cdots u_{k-1}100\in\Lambda^*_m,
\end{aligned}$$
$$\begin{aligned}
&u_1\cdots u_{k-2}1001\in\Lambda^*_m\Leftrightarrow u_1\cdots u_{k-2}10\in\Lambda^*_m\\
&\Rightarrow u_1\cdots u_{k-2}101, u_1\cdots u_{k-2}1000\in\Lambda^*_m
\end{aligned}$$
and
$$\begin{aligned}
&u_1\cdots u_{k-m+1}10^{m-1}1\in\Lambda^*_m\Leftrightarrow u_1\cdots u_{k-m+1}10\in\Lambda^*_m\\
&\Rightarrow u_1\cdots u_{k-m+1}101,\cdots, u_1\cdots u_{k-m+1}10^{m-2}1\in\Lambda^*_m\\
&\text{but } u_1\cdots u_{k-m+1}10^{m-1}0\notin\Lambda^*_m
\end{aligned}$$
respectively, recalling the definition of $\mu_p$.

Combining \textcircled{\footnotesize{$1$}} and \textcircled{\footnotesize{$2$}} we get $(1-p)(a_k-p^{m-2}d_{k-m+1})=c_k-p^{m-2}d_{k-m+1}$,
$$\text{i.e.,}\quad (1-p)a_k+p^{m-1}d_{k-m+1}=c_k \quad\text{for any }k\ge m.$$
That is,
$$(1-p)a_{k+m}+p^{m-1}d_{k+1}=c_{k+m} \quad\text{for any }k\ge0,$$
which implies
$$(1-p)\frac{1}{n}\sum_{k=0}^{n-1}a_{k+m}+p^{m-1}\frac{1}{n}\sum_{k=0}^{n-1}d_{k+1}=\frac{1}{n}\sum_{k=0}^{n-1}c_{k+m}.$$
Let $n\to\infty$, we get $(1-p)a+p^{m-1}d=c$. It follows from the same way that $pb+(1-p)^{m-1}c=d$.

Combining (1), (2) and (3) we get $a=\frac{p-p^m}{1-p-(1-p)^m}$.
\end{proof}

\begin{ack}
The author is grateful to Professor Jean-Paul Allouche for his advices on a former version of this paper, and also grateful to the Oversea Study Program of Guangzhou Elite Project (GEP) for financial support (JY201815).
\end{ack}


\begin{thebibliography}{00}

\bibitem[1]{A83} \textsc{J.-P. Allouche}, \textit{Th\'eorie des nombres et automates}, Th\`ese d'\'Etat, Bordeaux, 1983, http://tel.archives-ouvertes.fr/tel-00343206/fr/.

\bibitem[2]{ACS09} \textsc{J.-P. Allouche, M. Clarke, and N. Sidorov}, \textit{Periodic unique beta-expansions: the Sharkovski\u{\i} ordering}, Ergodic Theory Dynam. Systems 29 (2009), no. 4, 1055--1074.

\bibitem[3]{AC83} \textsc{J.-P. Allouche and M. Cosnard}, \textit{It\'erations de fonctions unimodales et suites engendr\'ees par automates}, C. R. Acad. Sci. Paris S\'er. I Math. 296 (1983), no. 3, 159--162.

\bibitem[4]{AC01} \textsc{J.-P. Allouche and M. Cosnard}, \textit{Non-integer bases, iteration of continuous real maps, and an arithmetic self-similar set}, Acta Math. Hungar. 91 (2001), no. 4, 325--332.

\bibitem[5]{AF09} \textsc{J.-P. Allouche and C. Frougny}, \textit{Univoque numbers and an avatar of Thue-Morse}, Acta Arith. 136 (2009), no. 4, 319--329.

\bibitem[6]{B89} \textsc{F. Blanchard}, \textit{$\beta$-expansions and symbolic dynamics}, Theoret. Comput. Sci. 65 (1989), no. 2, 131--141.

\bibitem[7]{BW14} \textsc{Y. Bugeaud and B.-W. Wang}, \textit{Distribution of full cylinders and the Diophantine properties of the orbits in $\beta$-expansions}, J. Fractal Geom. 1 (2014), no. 2, 221--241.

\bibitem[8]{C85} \textsc{M. Cosnard}, \textit{\'Etude de la classification topologique des fonctions unimodales}, Ann. Inst. Fourier (Grenoble) 35 (1985), no. 3, 59--77.

\bibitem[9]{DK95} \textsc{Z. Dar\'oczy and I. K\'atai}, \textit{On the structure of univoque numbers}, Publ. Math. Debrecen 46 (1995), no. 3-4, 385--408.

\bibitem[10]{DM46} \textsc{N. Dunford and D. S. Miller}, \textit{On the ergodic theorem}, Trans. Amer. Math. Soc. 60 (1946), 538--549.

\bibitem[11]{E49} \textsc{H. G. Eggleston}, \textit{The fractional dimension of a set defined by decimal properties}, Quart. J. Math., Oxford Ser. 20, (1949). 31--36.

\bibitem[12]{EJK90} \textsc{P. Erd\"os, I. Jo\'o, and V. Komornik}, \textit{Characterization of the unique expansions $1=\sum^{\infty} _ {i= 1} q^{-n_ i} $ and related problems}, Bull. Soc. Math. France 118 (1990), no. 3, 377--390.

\bibitem[13]{F90} \textsc{K. J. Falconer}, \textit{Fractal Geometry}, Mathematical foundations and applications. John Wiley $\&$ Sons, Ltd., Chichester, 1990. xxii+288 pp. ISBN: 0-471-92287-0

\bibitem[14]{F97} \textsc{K. J. Falconer}, \textit{Techniques in Fractal Geometry}, John Wiley $\&$ Sons, Ltd., Chichester, 1997. xviii+256 pp. ISBN: 0-471-95724-0

\bibitem[15]{F99} \textsc{G. B. Folland}, \textit{Real Analysis}, Modern techniques and their applications. Second edition. Pure and Applied Mathematics (New York). A Wiley-Interscience Publication. John Wiley $\&$ Sons, Inc., New York, 1999. xvi+386 pp. ISBN: 0-471-31716-0

\bibitem[16]{K14} \textsc{A. Klenke}, \textit{Probability theory}, A comprehensive course. Second edition. Translation from the German edition. Universitext. Springer, London, 2014. xii+638 pp. ISBN: 978-1-4471-5360-3; 978-1-4471-5361-0

\bibitem[17]{KKL17} \textsc{V. Komornik, D. Kong, and W. Li}, \textit{Hausdorff dimension of univoque sets and Devil's staircase}, Adv. Math. 305 (2017), 165--196.

\bibitem[18]{KL98} \textsc{V. Komornik and P. Loreti}, \textit{Unique developments in non-integer bases}, Amer. Math. Monthly 105 (1998), no. 7, 636--639.

\bibitem[19]{LLS19} \textsc{B. Li, Y.-Q. Li, and T. Sahlsten}, \textit{Random walks associated to beta-shifts}, arXiv:1910.13006 (2019).

\bibitem[20]{LW08} \textsc{B. Li and J. Wu}, \textit{Beta-expansion and continued fraction expansion}, J. Math. Anal. Appl. 339 (2008), no. 2, 1322--1331.

\bibitem[21]{L19} \textsc{Y.-Q. Li}, \textit{Hausdorff dimension of frequency sets in beta-expansions}, arXiv:1905.01481v2 (2019).

\bibitem[22]{LL18} \textsc{Y.-Q. Li and B. Li}, \textit{Distributions of full and non-full words in beta-expansions}, J. Number Theory 190 (2018), 311--332.

\bibitem[23]{P60} \textsc{W. Parry}, \textit{On the $\beta$-expansions of real numbers}, Acta Math. Acad. Sci. Hungar. 11 (1960), 401--416.

\bibitem[24]{R57} \textsc{A. R\'enyi}, \textit{Representations for real numbers and their ergodic properties}, Acta Math. Acad. Sci. Hungar. 8 (1957), 477--493.

\bibitem[25]{S97} \textsc{J. Schmeling}, \textit{Symbolic dynamics for $\beta$-shifts and self-normal numbers}, Ergodic Theory Dynam. Systems 17 (1997), no. 3, 675--694.

\bibitem[26]{W82} \textsc{P. Walters}, \textit{An Introduction to Ergodic Theory}, Graduate Texts in Mathematics, 79. Springer-Verlag, New York-Berlin, 1982. ix+250 pp. ISBN: 0-387-90599-5

\bibitem[27]{Y95} \textsc{K. Yosida}, \textit{Functional analysis}, Reprint of the sixth (1980) edition. Classics in Mathematics. Springer-Verlag, Berlin, 1995. xii+501 pp. ISBN: 3-540-58654-7

\end{thebibliography}
\end{document}